\RequirePackage[l2tabu, orthodox]{nag}

\documentclass[10pt,a4paper,reqno]{amsart}

\usepackage{lmodern}
\usepackage[T1]{fontenc}
\usepackage[utf8]{inputenc}
\usepackage[english]{babel}
\usepackage{microtype} 

\usepackage{amsmath,amsthm,amssymb,latexsym,mathtools,enumerate,caption,subcaption,hyperref}
\usepackage{epigraph,graphicx}

\usepackage[capitalize]{cleveref}

\usepackage{bbm}
\usepackage{todonotes}
\presetkeys%
    {todonotes}%
    {inline,backgroundcolor=yellow}{}

\graphicspath{ {./}{./figures/} }
\DeclareGraphicsExtensions{.png,.pdf}

\newtheorem{theorem}{Theorem}
\newtheorem{proposition}[theorem]{Proposition}
\newtheorem{lemma}[theorem]{Lemma}

\newtheorem{problem}[theorem]{Problem}

\newtheoremstyle{named}{}{}{\itshape}{}{\bfseries}{.}{.5em}{\thmnote{#3's }#1}
\theoremstyle{named}

\theoremstyle{definition}
\newtheorem{definition}[theorem]{Definition}

\newtheorem{remark}[theorem]{Remark}

\newcommand{\defin}[1]{%
\relax\ifmmode%
\textcolor{blue}{#1}%
\else \textcolor{blue}{\emph{#1}}%
\fi%
}

\newtheorem{THEO}{Theorem}

\newcommand{\bC}{\mathbb{C}}
\newcommand{\setRS}{\hat{\mathbb{C}}}

\renewcommand{\Re}{{\mathrm{Re}}}

\newcommand{\diffz}{\frac{d}{dz}}
\newcommand{\zeros}{\mathcal Z}

\newcommand{\minvset}[1]{{{\mathrm M}_{#1}^T}}

\DeclareMathOperator{\supp}{supp}

\numberwithin{equation}{section}

\numberwithin{theorem}{section}

\title[Holomorphic correspondences and Hutchinson invariant sets]
{Equidistribution of iterations of holomorphic correspondences and Hutchinson invariant sets}

\author[N.~Hemmingsson]{Nils Hemmingsson}
\address{Department of Mathematics,
Stockholm University,
S-10691, Stockholm, Sweden}
\email{nils.hemmingsson@math.su.se}

\begin{document}


\keywords{Holomorphic correspondences, invariant measures,
Hutchinson operators, invariant subsets of $\setRS$}
\subjclass[2020]{Primary 37F05, 37A05 Secondary 32H50}

\begin{abstract}
In this paper, we analyze a certain family of holomorphic correspondences on $\setRS\times\setRS$ and prove their equidistribution properties. In particular, for any correspondence in this family we prove that the naturally associated multivalued map $F$ is such that for any $a\in \bC$, we have that  $(F^n)_*(\delta_a)$  converges to a probability measure $\mu_F$ for which $F_*(\mu_F)=\mu_F d$ where $d$ is the degree of $F$. This result is used to show that the minimal Hutchinson invariant set in degree $n$ of a large class of operators and for sufficiently large $n$  exists and is the support of the aforementioned measure. We prove that under a minor additional assumption, this set is a Cantor set.
\end{abstract}

\maketitle

%

\tableofcontents

\section{Introduction}\label{sec:intro}
A \emph{holomorphic correspondence} on $\setRS\times\setRS$ is an algebraic curve defined by a polynomial in two variables. Consider a polynomial $P(z,w)$ and let $\Gamma_{P}\subset \bC^2$ be the plane algebraic curve given by $P(z,w)=0$. Compactify $\bC^2$ into $\setRS\times\setRS$ and consider the compactification $\overline{\Gamma_{P}}\subset\setRS\times\setRS$. From it we get a well-defined (multivalued) map $F : \setRS \to \setRS$, which sends any point $z_0 \in\setRS$ to the union of points in the second coordinate of the intersection of the horizontal line $\{(z,w)\subset \setRS\times\setRS: z=z_0\}$ with $\overline{\Gamma_{P}}$. We also call $F:z\to w$ a \emph{holomorphic correspondence} and it is the \emph{holomorphic correspondence defined by} $P(z,w)$. We denote by $F^\dagger:w\to z$ its adjoint, i.e. $z\in F^\dagger(w)\Longleftrightarrow (z,w)\in \overline{\Gamma_P}$.

In order to formulate the theorems in this text, we will need the definition of pushforward of currents by holomorphic correspondences.
Let now $ \Gamma\subset \setRS\times\setRS$ be an algebraic curve and $F$ be the correspondence induced by it, i.e. $$F(z)=\{w: (z,w)\in \Gamma \}. $$ Let us first suppose that $\Gamma$ contains no vertical nor horizontal lines, i.e. $$(z_0,\setRS)\not \subset \Gamma, (\setRS,w_0)\not \subset \Gamma.$$  We may decompose $\Gamma$ into its irreducible components:
$$\Gamma=\sum_{j=1}^N m_j \Gamma_j.$$
 Let $\pi_i:\setRS\times\setRS$ be the projection onto the $i$:th coordinate. The components $\Gamma_j$ define currents of integration on $\setRS\times\setRS$, $[\Gamma_j]$ of bi-degree $(1,1)$ and $[\Gamma]$ is the formal sum

$$[\Gamma]=\sum_{j=1}m_j[\Gamma_j].$$

For a current $S$ of bidegree $(0,0)$ or $(1,1)$ on $\setRS$, we define the pullback of $S$ by $F$ as

$$F_*(S)=(\pi_2)_*(\pi_1^*(S)\wedge [\Gamma])$$
whenever $(\pi_1^*(S)\wedge [\Gamma])$ makes sense. In particular, it makes sense when $S$ is a finite regular Borel measure (which may be viewed as a current of bidegree $(1,1)$), which is the focus of the present text. In particular, if $S=\delta_a$, then 

$$\langle F_*\delta_a,\phi\rangle =\sum_{j}m_j\sum_{(a,x)\in \Gamma_j}\phi(x)=\sum_{x\in F(a)} \phi(x),$$
where the points  $x\in F(a)$ are counted with multiplicities (see also \cite{Bharali2016TheDO} and \cite{mayuresh} for a more thorough introduction.)  In \cite{Dinh2006DistributionDV} you can find proper verification for the well-definedness of $F_*$. We also define $F^*=(F^\dagger)_*$.

The holomorphic correspondences we study in this text are as follows.
Fix a univariate polynomial $R_0(w)$ of degree $d$ with only simple zeros and a polynomial $P(z,w)$ of (total) degree at most $d$. Note that for any $d\in \mathbb N$, any bivariate polynomial in the variables $(z,w)$ of (total) degree at most $d$ can be uniquely written  as
$$P(z,w)=\sum_{j=0}^{d}P_j(w)(w-z)^{d-j},$$
with polynomials $P_j(w)$ such that $\deg P_j(w)\leq j$.
Let $$R_{\delta}(z,w)=R_0(w)+\delta P(z,w)$$
where $\delta $ is a (small) parameter. We study correspondences defined by $R_\delta(z,w)$ as above for $\delta$ of small absolute value. Denote by $F_\delta:z\to w$ the holomorphic correspondence defined by $R_\delta(z,w)=0$.  For sufficiently small $\delta$ we have $d=\deg_w(R_\delta(z,w))$ and this is the \emph{degree} of $F_\delta$.

When $R_\delta(z,w)$ is defined as above and $\Gamma$  is the algebraic curve defined by $R_\delta(z,w)=0$, it is easy to see that for sufficiently small $\delta$, $\Gamma$ contains neither vertical nor horizontal lines if and only if $\delta P(z,w)$ is not constant in $z$. If however $\delta P(z,w)$ is constant in $z$, then so is $R_\delta(z,w)$ and $$\Gamma=\{(\setRS,u_j): u_j \text{ is a zero of } R_\delta(z,w)\}.$$ That is, $F_\delta$ is the constant correspondence sending any point to the union of the zeros $u_j$ of $R_\delta(z,w)=R_\delta(z)$. In this case we agree to define $(F_\delta)_*(\delta_a)=\sum_{j=1}^{d}\delta_{u_j}$ for all $a$.

Recall that a set $S$ is \emph{forward invariant} (under $F_\delta$) if $F_\delta(S)\subset S$, \emph{backward invariant} if $F_\delta^\dagger(S)\subset S$ and \emph{completely invariant} if it is both forward and backward invariant. Unless otherwise stated, in this text the term measure always  refers to a regular Borel measure and $\delta_a$ denotes the Dirac measure at $a\in \setRS$. The \emph{minimal} set with a property $\mathcal P$, is (if it exists) the nonempty set $S_0$ with property $\mathcal P$ such that if $S_1$ also has property $\mathcal P$, then $S_0\subset S_1$. We are ready to state the first results of this text. In the theorems below $d\in \mathbb{N}$ is an arbitrary positive integer.

\begin{THEO}\label{thm:A}

Let $R_0(w)$ be any polynomial  of degree $d$ with only simple zeros, $P(z,w)$ any polynomial of (total) degree at most $d$ and let $F_{\delta}:z\to w$ be the holomorphic correspondence defined by $R_0(w)+\delta P(z,w)=0$ for any $\delta\geq 0$. For any $\epsilon>0$ there is $\Delta>0$ such that if $|\delta|<\Delta$, there is a probability measure $\mu_{F_\delta}$ on $\setRS$ with $(F_\delta)_*(\mu_{F_\delta})=\mu_{F_\delta}d$ such that for any $a\in \bC$,

$$\frac{1}{d^m}(F_\delta^m)_*(\delta_a) \xrightarrow{weak *}\mu_{F_\delta}, \quad \text{ as } m\to \infty,$$
and the support $\supp \mu_{F_{\delta}}$ is contained in the $\epsilon$-neighborhood of the zeros of $R_0(z)$. Moreover, $\supp \mu_{F_{\delta}}$  is the minimal closed under $F_{\delta}$ forward invariant set containing a point in $\bC$. 
\end{THEO}

\begin{remark} 
If for the minimal $j$ such that $P_j(w) \not \equiv 0$, we have that $Q_j(z)$ is not constant and $j<d$ ,  then  $a$ can be chosen equal to $\infty$, (see \cref{prop:earlier}).
\end{remark}
\begin{remark}
Note that $\Delta$ depends on $R_0(z)$ and $P(z,w)$.
\end{remark}

We can use \cref{thm:A} above to prove results in a related area known as P\'olya--Schur theory.
In this area, the following definition is central, see e.g. \cite{CravenCsordas2004, MR2521123}.

\begin{definition}\label{def0}
Given a linear operator $T:\bC[w]\to\bC[w]$ and a non-empty subset  $S\subset \bC$ of the complex plane, we say that
\emph{$S$ is a $T$-invariant set}, or that \emph{$T$ preserves $S$} if $T$ sends any polynomial with roots in $S$ to a
polynomial with roots in $S$ (or to $0$).
\end{definition}
The main question one studies is the following.

\begin{problem}\label{prob:2}
For a given set $S\subset \bC$, find all linear operators $T:\bC[w]\to\bC[w]$ that preserve $S$.
\end{problem}

Significant progress was made in this area about ten years ago and \cref{prob:2} has been solved for images of the unit disk under Möbius transformations and their boundaries \cite{MR2521123}, as well as for strips \cite{BrandenChasse2017}. However, for more general $S$, \cref{prob:2} is widely open.

The recent preprint \cite{ABS1} introduces the inverse of this problem, which can be stated as follows.
\begin{problem}\label{prob:general}
For a given linear differential operator $T$, find the $T$-invariant sets.
\end{problem}
In \cite{ABS1}, it is shown that any $T$-invariant set is convex (among other nice qualitative properties), but complete characterization seems to be difficult to obtain with current methods. In the same text, two specializations of the question above are formulated, which can be stated as follows. We begin with a definition.

\begin{definition}
For a given linear differential operator $T:\bC[w]\to\bC[w]$, a closed non-empty set $S\subset \bC$ is \emph{continuously Hutchinson invariant} if for every $z\in S$, the zeros of $T[(w-z)^t]$ for $t\geq 0$ lie in $S$.
\end{definition}
This leads to the following question.
\begin{problem}
For a given linear differential operator $T:\bC[w]\to\bC[w]$, find the continuously Hutchinson invariant sets. When they exist, find the minimal under inclusion continuously Hutchinson invariant set.
\end{problem}

In \cite{AHS}, this form of invariance was studied for the case when $T=Q_1(w)\diffz + Q_0(w)$. We managed to solve some of the questions we set out to. Among other results, we have shown that if at least one of the $Q_j(w)$ is non-constant, the minimal under inclusion continuously Hutchinson invariant set exists and found all invariant sets for which this is not the case. Further, we prove that the components of invariant sets are always simply connected and that the minimal continuously Hutchinson invariant set is equal to $\bC$ if and only if $|\deg Q_1-\deg Q_0|\geq 2$ or $\deg Q_1-\deg Q_0=1$ and in the expansion $$\frac{Q_1(w)}{Q_0(w)}=\lambda w+O(1) \quad \text{ as } w\to \infty, $$ we have $\Re(\lambda)<0$. 

Here we study a different specialization of \cref{prob:general}.

\begin{definition}\label{def:old}
For a given linear differential operator $T:\bC[w]\to\bC[w]$, a closed non-empty set $S\subset \bC$ is \emph{Hutchinson invariant} in degree $n$ if for every $z\in S$, the zeros of $T[(w-z)^n]$ lie in $S$.
\end{definition}
This notion leads to the following problem.
\begin{problem}\label{prob:main}
For a given linear differential operator $T:\bC[w]\to\bC[w]$ and $n\geq 1$, describe the Hutchinson invariant sets in degree $n$. When it exists, find the minimal under inclusion Hutchinson invariant set in degree $n$.
\end{problem}

\begin{definition}
For a given linear differential operator $T:\bC[z]\to\bC[z]$,
we denote by $\minvset{H,n}$ the minimal under inclusion Hutchinson invariant set degree $n$, when it exists.
\end{definition}
\begin{remark}
Any $\minvset{H,n}$ is contained in any $T$-invariant set.
\end{remark}
We will often say that a set is \emph{Hutchinson invariant} instead of \emph{Hutchinson invariant in degree $n$}. This will refer to the same thing and $n$ will be clear from the context.
To state the results of the present text, we also need the following definitions

\begin{definition}\label{def:degNonDegFuchs}
Given a differential operator of finite order $k\geq 1$ 
\begin{equation}\label{eq:1}
T = \sum_{j=0}^k Q_j(w)\frac{d^j}{dw^j}
\end{equation} 
with polynomial coefficients, we say that $T$ is \emph{non-degenerate} if
\[
\max_{0\le j\le k} (\deg(Q_{j})-j) =\deg(Q_k)-k.
\]
\end{definition}
If $\max_{0\le j\le k} (\deg(Q_{j})-j) =0$, we say that $T$ it is \emph{exactly solvable}.
The terminology above is in line with \cite{AHS,ABS1,ABS2}. In the literature about holomorphic correspondences, non-degeneracy of a correspondence defined by $P(z,w)=0$ may also refer to the fact that every non-constant factor of $P(z,w)$ involves both $z,w$. This is not what we mean in this text. 

We denote for $n\geq k$ by $T_n:z\to w$ the holomorphic correspondence defined by $\frac{T[(w-z)^n]}{(w-z)^{n-k}}=0$. We will from here on out always assume that $n\geq k$.
If there is a minimal set $S$ with the property that it is closed, contains a point in $\bC$ and is forward invariant set under $T_n$, then $S\setminus {\infty}\subset \bC$ is equal to $\minvset{H,n}$
A fixed point $z$ of $T_n$ is a point such that $z\in T_n(z)$. We also define a periodic point to be such that $z\in T_n^l(z)$ for some $l\geq 1$. If $z$ is a periodic point and there is a local branch $w_I(z)$ of $T_n^l$ such that $w_I(z)=z$ and $w_I'(z)<1$, then we say that $z$ is an attracting periodic point.  We are now ready to state the second result of the text.

\begin{THEO}\label{thm:C}
Suppose that $T$ given as in \eqref{eq:1} is non-degenerate and all zeros of $Q_k(w)$ are simple. For any $\epsilon>0$, there is $N$ such that if $n\geq N$, then
\begin{itemize}
\item
 $T_n$ has the properties of $F_\delta$ in \cref{thm:A} and $\minvset{H,n}$ equals the support of the measure $\mu_{T_n}$.
\item  Each point in $\minvset{H,n}$ is the limit of a sequence of attracting periodic points of $T_n$ and is contained in the $\epsilon$-neighborhood of the zeros of $Q_k(w)$.
\item  If $k\geq 2$ and at least one $Q_j(w)$, $j\neq k$ is not identically zero then $\minvset{H,n}$ is a Cantor set.
\end{itemize}
\end{THEO}
\begin{remark}
It is sufficient that one zero of $Q_k(w)$ is simple for $\minvset{H,n}$ to exist (see \cref{prop:prop1})
\end{remark}

\medskip
The structure of the paper is as follows. In \cref{sec:prel} we provide some preliminaries on iterations of holomorphic correspondences and in its subsection \cref{sec:equi} we provide important results concerning the equidistribution of (pre-)images of rational maps and correspondences. \cref{sec:proofsa} contains the proof of \cref{thm:A} and \cref{sec:proofsb} contains the proof of \cref{thm:C}. Lastly \cref{sec:final} contains some final remarks and some unsolved problems related to this text. 

 \medskip
\noindent
\textbf{Acknowledgements.}
I want to thank B.~Shapiro for suggesting this topic and always readily answering any questions I had. I want to thank M.~Lyubich for his interest, support and several relevant references. Furthermore, I want to thank P.~Alexandersson for letting me discuss and check my ideas with him.

\section{Preliminaries on iterations of holomorphic correspondences}\label{sec:prel}

Iteration of holomorphic correspondences is a subject that was mentioned already by P.~Fatou, but in which far fewer results have been obtained so far compared to in classical dynamics studying iterations of single-valued functions. 

In \cite{doi:10.1142/S0218127491000592}, the authors studied iterations of holomorphic correspondences with the help of a somewhat special construction. Instead of using the base space $\setRS$, they use the space of orbits of an iteration, and show that similarly to classical dynamics, there is a decomposition of the space of orbits in completely invariant subsets, one of which is ``regular'' and the other is more ``chaotic'' and highly inferred with branch points.

The text \cite{BullettPenrose2-1994} is an excellent exposition about the iterations of holomorphic correspondences and presents the foundations of the theory. Many of the results in this area have been obtained by S.~Bullett, jointly with C.~Penrose and L.~Lomonoco. They have found interesting examples of what they call \emph{matings} between quadratic polynomials and the modular group \cite{Bullett1988, BullettPenrose1994, BullettPenrose2-1994, Bullett2010, BullettLomonaco}. In \cite{BullettPenrose2-1994} they also define multiple different sets regarding the forward and backward iterations of holomorphic correspondences and conjecture different properties of these sets. The text \cite{komigennu} is an excellent survey of the results mentioned above. 

Recently \cite{siqueira_2021}, established geometric rigidity of certain holomorphic correspondences in the family $(w-c)^q=z^p$. Moreover, the paper \cite{LeeLyubichMakarov} studied anti-holomorphic correspondences and found matings similar to that of S.~Bullett and C.~Penrose.

\subsection{Equidistribution of images for rational maps and correspondences}\label{sec:equi}
In this section, we go through important results concerning equidistribution of (pre-)images of rational maps and correspondences and some related results.
In 1983 M.~Lyubich proved the following important generalization of Brolin's theorem.

\begin{theorem}[\cite{ljubich_1983}]\label{th:misha}
Suppose that $R:\setRS\to \setRS$ is a rational map of degree $d\geq 2$. Then there is a probability measure $\mu$ supported on the Julia set of $R$ with the property that $\mu(E)=\mu(f^{-1}(E))$ for any Borel set $E\subset \setRS$ and 
$$\frac{1}{d^n}\sum_{R^n(z)=a}\delta_z \xrightarrow{weak *} \mu \quad \text{as $n\to \infty$}$$
for all $a\in \setRS \setminus \mathcal E$ where $|\mathcal E|\leq 2.$
\end{theorem}

We turn to stating the results similar to \cref {th:misha} but where instead of dealing with rational maps, they concern correspondences.  A \emph{rational semigroup} is a semigroup consisting of non-constant rational maps on $\setRS$ with the semigroup operation being function composition. 

\begin{theorem}[\cite{boyd}]
Let $G$ be a finitely generated rational semigroup. Then there is a probability measure $\mu$ on $\setRS$ such that for $a\in \setRS \setminus \mathcal E$ where $|\mathcal E|\leq 2,$

$$\frac{1}{d^n}\sum_{g(z)=a, l(g)=n}\delta_z \xrightarrow{weak *} \mu \quad \text{as $n\to \infty$}.$$
 Furthermore, the support of $\mu$ is the Julia set of $G$, which is the closure of the Julia set of all maps in $G$.

\end{theorem}

Note that any finitely generated rational semigroup $\langle f_1,f_2,...,f_n\rangle$ corresponds to a algebraic correspondence as 
$$\prod_{j=1}^n(f_j(z)-w).$$

We suppose now that we are dealing with a correspondence $F:\setRS\to \setRS$ defined by a polynomial $P(z,w)$ of bidegree $(d_0,d_1)$  not containing factors of the form $(w-w_0)$ nor $(z-z_0)$. The numbers $d_1$ and $d_0$ are the number of images of a point under $F$ and $F^\dagger$ respectively, counting multiplicity. The assumption that $P(z,w)$ does not contain factors of the form $(w-w_0)$ or $(z-z_0)$ is not necessary in all of the statements below but makes the presentation much clearer. Moreover, in several cases the results below generalize to complex manifolds other than the Riemann sphere, but we here state them in this setting and refer the curious reader to the original papers for the full statements.  A function $\phi:\setRS\to [\infty,\infty)$ is \emph{quasi-p.s.h.} if it can be locally written as a sum of a smooth and a plurisubharmonic function. A proper subset $A$ of $\setRS$ is \emph{pluripolar} if it is contained in $\{z\in \setRS: \phi(z)=-\infty\}$ for some quasi-p.s.h. function $\phi$. A special case of two theorems in \cite{Dinh2020DynamicsOH} amounts to the following.

\begin{theorem}[\cite{Dinh2006DistributionDV}]\label{dinhsibony}
Suppose that $1\leq d_0<d_1$ and let $\omega_{FS}$ be the normalized Fubini-Study form. Then there is a measure $\mu_F$ such that $F_*(\mu_F)=d_1\mu_F$ and  
$$\frac{1}{d_1^n}(F^n)_*\omega_{FS}  \xrightarrow{weak *} \mu_F \quad \text{as $n\to \infty$}.$$  There is a pluripolar set $\mathcal E\subsetneq\setRS$ such that if $a\notin \mathcal E$, then

$$\frac{1}{d_1^n}(F^n)_*\delta_a  \xrightarrow{weak *} \mu_F \quad \text{as $n\to \infty$}.$$

\end{theorem}

If $F$ is a holomorphic correspondence induced by the algebraic curve $\Gamma$, then it is \emph{weakly modular} if there is a positive measure $\mu$ on $\Gamma$ and probability measures $\mu_1,\mu_2$ on $\setRS$ such that $\mu=(\pi{_1}_{|\Gamma})^*\mu_1$ and $\mu=(\pi{_2}_{|\Gamma})^*\mu_2$. 

\begin{theorem}[\cite{Dinh2020DynamicsOH}]
Suppose that $F$ is a non-weakly modular correspondence and $2\leq d_0=d_1$. Then there are two measures $\mu_\pm $ such that $F^*(\mu_+)=d_1\mu_+$ and $F_*(\mu_-)=d_1\mu_-$ and if $\alpha$ is any smooth $(1,1)$-form on $\setRS$ with $\int_{\setRS} \alpha=1$, then
$$\frac{1}{d_1^n}(F^n)^*(\alpha)\xrightarrow{weak *}\mu_+\text { and  }\frac{1}{d_1^n}(F^n)_*(\alpha)\xrightarrow{weak *}\mu_-.$$
If no critical value of $F^\dagger$ is periodic, then there is a constant $0<\lambda<1$ such that for any $a\in \setRS$ and test function $\phi$ of class $C^\beta$ with $0<\beta\leq 1$, we have 

$$\left|\left\langle \frac{1}{d_1^{n}}(F^n)_*(\delta_a)-\mu_-,\phi\right\rangle\right|\leq A_\beta\|\phi\|_{C^\beta}\lambda^{\beta n}$$
where $A_\beta>0$ is a constant independent of $n,a$ and $\phi $.
\end{theorem}
Of course, if no critical value of $F$ is periodic, then the analogous statement for $\mu_+$ holds.
To state the next result, we need a few notions. The $\omega$-limit set of a set $S$, $\omega(S)$ may be defined as 
\[\bigcap \{K: K \text{ is a closed forward invariant set such that } F^n(S)\subset K \text{ for some } n\geq 0\}.\]
 
Next, we say that $\mathcal A\subset \setRS$ is an \emph{attractor} if there is an open set $U$ of $\mathcal A$ such that $\omega(U)=\mathcal A$.
Moreover, we say that $\mathcal A$ is a \emph{strong attractor} if it is an attractor and if there is a point $a_0\in \mathcal A$  such that for any $u_0\in U$ there is a sequence $(u_n)$ such that $u_n\in F^n(u_0)$ and $u_n\to a_0$ as $n\to \infty$.

\begin{theorem}[\cite{Bharali2016TheDO}]\label{th:barali}
Suppose that $F$ is such that $d_0\leq d_1$, $F$ has a strong attractor $\mathcal A$ that is disjoint from the set of critical values of $F^\dagger$. Then there is a probability measure $\mu$ and an open set $U\supset \mathcal A$ such that $F_*(\mu)=d_1 \mu $ such that 

$$\frac{1}{d_1^n}(F^n)_*\delta_x  \xrightarrow{weak *} \mu \quad \text{as $n\to \infty$},\quad  \forall x\in U.$$

\end{theorem}
One should note that the definitions of \emph{critical point} in \cite{Dinh2020DynamicsOH} and \cite{Bharali2016TheDO} are slightly different and we refer the reader to the two papers for the precise definitions. In \cite{Bharali2016TheDO} they also show that the support of the measure $\mu_F$ in \cref{dinhsibony} is disjoint from the \emph{normality set} of $F^\dagger$, see \cite{Bharali2016TheDO} for the of the definition normality set. Lastly, we state a theorem relating minimal invariant sets and the equidistribution of images.

\begin{theorem}[\cite{mayuresh}]\label{th:mayuresh}
Suppose that $\mu_F$ is the measure appearing in \cref{dinhsibony}. Then $\supp(\mu_F)$ is the minimal closed non-pluripolar forward invariant set.

\end{theorem}

\section{Proof of \cref{thm:A}}\label{sec:proofsa}
On any compact set $K$ such that $F_\delta(K)\subset K$ we define the operator $A_K:C(K)\to C(K)$ as

$$A_K(\phi)(\xi):=\frac{1}{d}\langle F_*\delta_a,\phi\rangle=\frac{1}{d}\sum_{\zeta\in F_\delta(\xi)} \phi(\zeta),$$
where we take multiplicities into account and $d$ is the degree of $F_\delta$. For $K=\setRS$, we denote $A_{\setRS}$ by $A:C(\setRS)\to C(\setRS)$, i.e.	

$$A(\phi)(\xi)=\frac{1}{d}\sum_{\zeta\in F_\delta(\xi)} \phi(\zeta),$$
where we once again take multiplicities into account. We say that $A_K$ is \emph{almost periodic} if for any $\phi\in C(K)$, we have that $\{A_K^m\phi\}_{m=1}^\infty$  is strongly conditionally compact. In particular, if for any $\phi\in C(K)$, we have that $\{A_K^m \phi \}_{m=1}^\infty$ is an equicontinuous family, then it is strongly conditionally compact.

We will need the following statements from \cite{ljubich_1983}.

\begin{proposition}[\cite{ljubich_1983}]\label{prop:misha1}
Suppose that $A: \mathfrak B\to \mathfrak B$ is an almost periodic operator on a complex Banach space $\mathfrak B$ and that its unitary spectrum is $\{1\}$ and $1$ is a simple eigenvalue. Suppose that $h\neq 0$ is an invariant vector of $A$. Then there is a linear functional $\mu$ such that $\mu\circ A=\mu, $ $\mu(h)=1$ and $$A^n(v)\to \mu(v) h, \quad \text{ as } n\to \infty,$$ for each $v\in \mathfrak B$.
\end{proposition}

We denote by $D(z,r)$ the closed disk in $\bC$ of radius $r$  centered at $z$ and $d(\cdot,\cdot)$ denotes the spherical metric on $\setRS$. $C(K)$ is the space of continuous functions on $K$ with image in $\bC$. It is a Banach space with norm $\|\phi\|_K=\sup_{z\in K}|\phi(z)|$.

For a fixed degree $d$ polynomial $R_0(z)$ with only simple zeros and fixed $$P(z,w)=\sum_{j=0}^{d}P_j(w)(w-z)^{d-j},$$
with polynomials $P_j(w)$ such that $\deg P_j(w)\leq j$ of degree at most $d$ we shall denote by $R_\beta$ the polynomial

 \begin{equation}\label{eq:def1}R_\beta(z,w)=R_0+\sum_{j=0}^{d}\beta_jP_j(w)(w-z)^{d-j}\end{equation}
where $\beta=(\beta_0,\beta_1,...,\beta_{d}).$ If we can prove the properties in \cref{thm:A} for the correspondences defined by $R_\beta(z,w)$ for $\beta_j$ of sufficiently small absolute value, we are done.

Now, we define $S(w):=R_0(w)+\beta_{d}P_{d}(w)$ and $a_k$ as the coefficient of $w^{k}$ in $S(w)$. We find $\beta_{d,0}$ such that if $|\beta_d|<\beta_{d,0}$ then $S(w)$ has only simple zeros, these zeros are of uniformly positive distance and $|a_d|$ is uniformly positive. As $R_0(w)$ has only simple zeros and $\deg P_{d}(w)\leq \deg R_0(w)=d$, this is possible. 
We start by proving the following lemma.
\begin{lemma}\label{le:1}
Let $R_\beta(z,w)$ be as in \cref{eq:def1}. There exist $\beta_{j,0}>0, M_0>0$, such that if $M\geq M_0, |\beta_j|< \beta_{j,0}$ and $z\notin D(0,M)$, then for any point $w\in F_\delta(z)$, we have $|w|<|z|/2$.
\end{lemma}
\begin{remark}
\cref{le:1} implies that for any $v_0\neq \infty$, there exists some finite $m$, (depending on $v_0$) such that $(F_\delta)^m(v_0)\subset D(0,M)$ provided $\delta$ is small enough.
\end{remark}
\begin{proof}[Proof of \cref{le:1}] 
The equation we are studying is
\begin{equation}\label{eq:simplifying}S(w)+\sum_{j=0}^{d-1}\beta_j P_j(w)(w-z)^{d-j}=0.\end{equation} 

For fixed $P_j(z)$, $j<d$ and $S$, by investigating the leading terms of \eqref{eq:simplifying}, we get that there exist $M_1>0$ and $C>0$ both independent of all $\beta_j$ such that if and $|z|\geq M_1$ and $(z,w)$ are solutions to \eqref{eq:simplifying}, then 

\begin{equation}\label{eq:simplifying2}1\leq C \sum_{j=0}^{d-1}|\beta_j| \left|\frac{w-z}{w}\right|^{d-j}\frac{1}{|w|^{d_j}}\end{equation}
where $d_j\geq 0$ equals $j-\deg Q_j\geq 0$. Note that we use uniform positivity of $|a_d|$ here.

Take $M\geq M_1$ and suppose now that $z\notin D(0,M)$ with $z\neq \infty$ such that $w\geq|z|/2$ and $w\in F_\delta(z)$. Then by \eqref{eq:simplifying2}, 

$$1\leq C \sum_{j=0}^{d-1}|\beta_j| \left|\frac{w-z}{w}\right|^{d-j}\frac{1}{|w|^{d_j}}\leq C  \sum_{j=0}^{d-1}\left|\beta_j\right| 3^{d-j},$$
This is clearly false for $\beta_j$, $j=0,...,d-1$ of sufficiently small absolute value. In particular there is  $\beta_{j,0}>0$ such that if $|\beta_j|<\beta_{j,0}$ for all $j$, the statement is false. This argument implies the statement with the mentioned $\beta_{j,0}$ ($\beta_{d,0}$ was defined before the statement of the lemma) and $M_0=2M_1$.
\end{proof}

\begin{remark}
Note that $M_0$ and $\beta_{j,0}$ depends on $P_j$ and $R_0$ but, importantly, that $M_0$ is independent of $\beta_{j,0}$  provided $\beta_{j,0}$ are sufficiently small.
\end{remark}

We turn to proving \cref{thm:A}.

\begin{proof}[Proof of \cref{thm:A}]
Let $u_1,...,u_{d}$ denote the zeros of $S(w)$. Choose $\eta_0>0$ such that the $\eta_0$-neighborhoods of the zeros of $S$ do not overlap and such that for all $\beta_d$ for which $|\beta_{d}|<\beta_{d,0}$, we have $|S'(w)|>\lambda \neq 0$ provided $|z-u_1|<\eta_0$. This is possible for all $\beta_k$ since the distance of the zeros of $S$ are uniformly positive. We then define $M\geq M_0$ from \cref{le:1} to be large enough so that for all $\beta_d$ such that $|\beta_{d}|<\beta_{d,0}$, $D(0,M)$ contains the $\eta_0$-neighborhoods of the zeros of $S$. This is possible since $a_d$ is uniformly positive so the zeros of $S(w)$ are contained in some disk not containing $\infty$ independently of $\beta_d$ with $|\beta_d|<\beta_{d,0}$. From now on, $M$ and the parameter $\beta_d$ are fixed, but we may decrease the absolute values of $\beta_j$.

Next we note that for any given $\delta_1$, if $|\beta_j|$, $j<d$ are chosen sufficiently small, then for each $u_j$ and for any $z\in D(0,M)$ there is a point in $F_\delta(z)$  which is of less than $\delta_1$-distance from $u_j$. This is because for $z\in D(0,M)$, the roots of $S(w)$ are the limit points of the set of solutions to
$$S(w)+\sum_{j=0}^{d-1}\beta_jP_j(w)(w-z)^{d-j}$$
as $\beta\to (0,0,...,0,\beta_d)$.
Indeed, since $d\geq \deg( P_j(w)w^j)$ there are $d$ solutions for sufficiently small $|\beta_j|$, and each term except $S(w)$ tends uniformly to $0$ as $\beta\to(0,0,...,0,\beta_d)$ for each $z$ in any compact set not containing $\infty$.
We pick $\beta_j$, $j<d$ such that $|\beta_j|$ are small enough so that for any $z\in D(0,M)$, all points in $F_\delta(z)$ are of less than $\eta_0$-distance from the zeros of $S(w)$. 

For each $z\in D(0,M)$, there is only one point in $F_\delta(z)$ that belongs to $\{z:|z-u_l|<\eta_0\}$ for each $l$. Thus, locally we can define $w_l(w)$ to be the holomorphic branch of $F_\delta(z)$ such that $w_l(u_l)=u_l$ (note that $u_l$ are fixed points of $F_\delta$). In particular we have that $w_l(z)$ is defined for $z\in D(0,M)$. The next step is to find the derivative $w_l'(z)$. To simplify our work with the summations, we let $S_j(w)=P_j(w)$ and $c_j=\beta_j$ for $j<d$ and $S_d(w)=S(w)$ and $c_d=1$. For any $l$, differentiating \cref{eq:simplifying} implicitly with respect to $z$ yields
\begin{equation}\label{eq:deriv}
\sum_{j=0}^{d}\left(c_j(d-j) S_j(w_l(w))(w_l(w)-w)^{d-j-1}(w_l'(w)-1)\right)
\end{equation}

\begin{equation*}
+\sum_{j=0}^{d}c_j w_l'(z)S_j'(w_l(z))(w_l(z)-z)^{d-j}=0.
\end{equation*}
Suppose now that $z\in D(0,M)$ and let $l=1$. Let us solve \cref{eq:deriv} for $w_1'(z)$. We obtain

\begin{equation}
w_1'(z)=\frac{\sum_{j=0}^{d-1}c_j(d-j) S_j(w_1(z))(w_1(z)-z)^{d-j-1}}{\sum_{j=0}^{d}\left(c_j(d-j) S_j(w_1(z))(w_1(z)-z)^{d-j-1}+c_jS_j'(w_1(z))(w_1(z)-z)^{d-j}\right)},
\end{equation}
or equivalently,

\begin{equation}\label{eq:deriv2}
w_1'(z)=\frac{1}{1+g_{\beta,1}(z)},
\end{equation}
where 

$$g_{\beta,1}(z)=\frac{\sum_{j=0}^{d}c_jS_j'(w_1(z))(w_1(z)-z)^{d-j}}{\sum_{j=0}^{d-1}c_j(d-j) S_j(w_1(z))(w_1(z)-z)^{d-j-1}},$$ if
$\sum_{j=0}^{d}c_j(d-j) S_j(w_1(z))(w_1(z)-z)^{d-j-1}\neq 0$ and $w_1'(z)=0$ otherwise.

Recall that we defined $\eta_0>0$ to be such that $|S_d(w)|=|S(w)|>\lambda \neq 0$ if $|z-u_j|<\eta_0$ for some $j$. Since $|w_1(z)-u_1|<\eta_0$ if $z\in D(0,M)$, it follows that $g_{\beta,1}(z)\to \infty$ as $\beta\to (0,...,\beta_d)$ for all $z\in D(0,M)$. In particular, for sufficiently small $|c_j|=|\beta_j|$, $j<d$, we have $|g_{\beta,1}(z)|>3$ for all $z\in D(0,M)$. Thus $|w_1'(z)|<\frac{1}{2}$ and since $w_1(u_1)=u_1$, it follows that $|w_1(z)-u_1|<\frac{1}{2}|z-u_1|$. In particular, for any $z\in D(u_1,\eta_0)$ we have that $w_1^n(z)\to u_1$ as $n\to \infty$. Together with \cref{le:1}, we obtain that for any $z\in \bC$ there is a sequence $(b_j)_{j=1}^\infty$, $b_j\in \bC$ such that $b_j\in (F_\delta)^j(z)$ and $\lim_{j\to \infty}{b_j}=u_1$. We can thereby prove the following statement from \cite{ljubich_1983} precisely the same way that was originally done (see also  \cite{Bharali2016TheDO}). We omit the proof here.
 \begin{lemma}[\cite{ljubich_1983}]\label{le:misha}
For any $K$ not containing $\infty$, the unitary spectrum of $A_K$ is $\{1\}$ and the corresponding eigenspace consists only of the constant functions. 
\end{lemma}

The next part of the proof also uses the ideas found in \cite{ljubich_1983}. It is also these ideas that are used in the proof of \cref{th:barali} in \cite{Bharali2016TheDO} and the proof given below is similar to theirs.

So, let $\phi\in C(D(0,M))$ and $\epsilon_0>0$ be given. For $I=i_1i_2i_3,...,i_\ell$, $1\leq i_j \leq d$ and $z\in D(0,M)$, define $w_I(z)$ by $w_{i_1}\circ w_{i_2}\circ\cdots w_{i_\ell}(z)$. Let $\ell(I)$ denote the length of $I$. Take $\nu>0$ such that if $d(\xi,\zeta)<\nu$ and $\xi,\zeta\in D(0,M)$ then $|\phi(\xi)-\phi(\zeta)|<\epsilon_0$. Now, $\mathcal F=\{w_I(z)\}$ is a normal family in $D(0,M)$ by Montel's theorem, since any map in $\mathcal F$ maps $D(0,M)$ to itself. Hence, there is $r>0$ such that if $d(\xi,\zeta)<r,$ and $\xi,\zeta\in D(0,M)$ then $|w_I(\xi)-w_I(\zeta)|<\nu$ for any $w_I\in \mathcal F$. Define $\mathcal F_N=\{(w_I,I): w_I\in \mathcal F, \ell(I)=N\}$. Note that $|\mathcal F_N|=d^N$. We find that for all $N\geq 0$,

$$|A^N_{D(0,M)}\phi(\xi)-A^N_{D(0,M)}\phi(\zeta)|=\frac{1}{d^N}\sum_{(w_I,I)\in F_N}|\phi(w_I(\zeta))-\phi(w_I(\xi))|<\epsilon_0,$$
provided $d(\xi,\zeta)<r$ and $\xi,\zeta\in D(0,M)$. Next, suppose that $K\subset \setRS$ is a compact set not containing $\infty$ and let $\phi\in C(K)$. Take $\nu>0$ such that if $d(\xi,\zeta)<\nu$ and $\xi,\zeta\in K$ then $|\phi(\xi)-\phi(\zeta)|<\epsilon_0$. By \cref{le:1}, there is an $M$ such that $(F_\delta)^M(\zeta)\subset D(0,M)$ for any $\zeta \in K$. Now, since the zeros of a polynomial depend continuously on the coefficients and $M$ is finite, we can find $\eta>0$ such that if $d(\xi,\zeta)<\eta$, then $d_H((F_\delta)^m(\xi),(F_\delta)^m(\zeta))<\min(\nu, r)=\kappa$ for $m=1,2,..., M$, where $d_H$ denotes the spherical Hausdorff distance. We can order the points in $(F_\delta)^m(\xi)$ and $(F_\delta)^m(\zeta)$ as $\xi_i$, $\zeta_i$, $i=1,...,d^m$, taking multiplicity into account, so that $d(\xi_i,\zeta_i)<\kappa$. Suppose $N\leq M -1$. Then

$$|A^N_{D(0,M)}\phi(\xi)-A^N_{D(0,M)}\phi(\zeta)|=\frac{1}{d^N}\sum_{i=1}^{d^N}|\phi(\zeta_i)-\phi(\xi_i)|<\epsilon_0,$$
if $|\xi-\zeta|<\eta$. If $N\geq  M$, then all $\zeta_i$ and $\xi_i$ belong to $D(0,M)$ and so

$$|A^N_{D(0,M)}\phi(\xi)-A^N_{D(0,M)}\phi(\zeta)|=\frac{1}{d^M}\sum_{i=1}^{d^M}\frac{1}{d^{N-M}}\sum_{\ell(I)=N-M}|\phi(w_I(\zeta_i))-\phi(w_I(\xi_i))|<\epsilon_0.$$

Since $|A^m_K(\phi)|\leq\|\phi(z)\|$, we have proved that for any $K$ not containing $\infty$, $A_K$ is almost periodic. Since the function $\mathbbm 1$ that is identically equal to $1$ is an eigenvector of $A_K$, by \cref{prop:misha1} there exists a complex regular measure $\mu_K$ on $K$ such that $$\left\|A_K^m(\phi)-\left(\int \phi \: d\mu_K\right)\mathbbm 1\right\|_K\to 0,\quad \text{ as }m\to \infty,$$
and it is clear that $\mu_K$ is a probability measure on $K$.
Note that since $|A^m_K(\phi)|\leq \max_{z\in K}|\phi(z)|$ we get that $\mu$, the functional in \cref{prop:misha1}, is bounded, hence continuous. Therefore we have an integral representation as above by the Riesz-Markov theorem.

Next, fix a compact $K\subset \setRS$ not containing $\infty$. We still assume that $F_\delta(K)\subset K$. Let us consider $\mu_K$ as a measure on $\setRS$ (it is 0 outside of $K$). We shall now prove that the measure $\mu_K$ is independent of $K$.  We have that for any $z\in K$,  $$u_1\in \overline{\bigcup_{m=0}^\infty (F_\delta)^m(z)}\subset K.$$ Given $\phi\in C(\setRS)$ denote by $\phi_K$ its restriction to $K$. We have that 

$$\int\phi \:d\mu_K=\int \phi_K\:d\mu_K=\lim_{m\to \infty}A^m_K(\phi_K(u_1))=\lim_{m\to \infty}A^m(\phi(u_1)),$$
showing independence of $K$.  Next,

$$\int A(\phi(z))d\mu=\int A(\phi_k(z))d\mu_k=\int \phi_K(z)d\mu_K=\int \phi(z)d\mu,$$
where the second inequailty comes from the fact that $\mu_K$ is $A_K^*$-invariant by \cref{le:misha}. This shows that $\mu$ is $A^*$-invariant, i.e. $$(F_\delta)_*(\mu)=\mu d.$$

If $F_\delta(K)\not\subset K$, replace $K$ with a compact $L\supset K$ not containing $\infty$ such that $F_\delta(L)\subset L$, which is possible by \cref{le:1}. Then

$$\left\|A^m\phi-\int \phi \:d\mu\right\|_K\leq \left\|A^m\phi-\int \phi\: d\mu\right\|_L=\left\|A_L^m\phi_L-\int \phi_L d\mu_L\right\|_L\to 0$$

as $m\to \infty$. Thus there is a measure $\mu$ on $\setRS$ such that 

\begin{equation}\label{eq:concl}\left\|A^m\phi-\int \phi \:d\mu\right\|_K\to 0,
\end{equation}
for any compact $K$ not containing $\infty$. We denote this measure by $\mu_{F_\delta}$ and recall that  
\begin{equation}\label{eq:invariant}(F_\delta)_*(\mu_{F_\delta})=\mu_{F_\delta}d.\end{equation} Now, take $a\neq \infty$, set $K=\{a\},$ and define $\mu_{z,m}=\frac{1}{d^m}(F_\delta^m)_*(\delta_a).$ Note that for any $\phi \in C(\setRS),$ $$\int \phi(z)d\mu_{z,m}=A^m \phi(z).$$
Using \eqref{eq:concl}, we find that 
$$\int \phi(z)d\mu_{z,m}=A^m \phi(z) \to \int\phi\: d\mu_{F_\delta},\quad  \text{ as }m \to \infty.$$
This proves the first part of \cref{thm:A}.

We turn to the second part. By \eqref{eq:invariant}, $\supp \mu_{F_\delta}$ is forward invariant, and of course it is closed. Let $S\neq \{\infty\}$ be a closed and non-empty forward invariant set. We have that for any $a\in S\setminus \infty$, $$\frac{1}{d^m}(F_\delta^m)_*(\delta_a) \xrightarrow{weak *}\mu_{F_\delta}$$
so $\supp(\mu_{F_\delta})$ is contained in $S$. It follows that  $\supp \mu_{F_\delta}$ is the minimal closed under $ {F_\delta}$ forward invariant set containing a point in $\bC$. We turn to proving that for any $\delta>0$ there is $\epsilon_\delta>0$ such that $\supp \mu_{F_{\delta}}$ is contained in the $\delta$-neighborhood of the zeros of $R_0(w)$. We choose $\beta_d$ small enough so that it fulfills the criterion above and such that the zeros of $S(w)$ are contained in the $\delta/2$-neighborhood of the zeros of $R_0(w)$. Next, we pick $\eta_0$ small enough to fulfill the criteria used in the proof and such that $\eta_0<\delta/2$. It is easy to see from the proof above that  $\supp \mu_{F_{\delta}}$ is contained in the $\eta_0$-neighborhood of the zeros of $S(w)$ so the conclusion of the statement is clear.  
This concludes the proof of \cref{thm:A}.

\end{proof} 

Under some extra assumptions, $a$ can be chosen equal to $\infty$:

\begin{proposition}\label{prop:earlier}

Let $R_0(w)$ be any polynomial of degree $d$ with only simple zeros. Let $P(z,w)$ by any polynomial of (total) degree at most $d$, such that for the minimal $j$ such that $P_j(w) \not \equiv 0$, $Q_j(z)$ is not constant and $j<d$. 

Let $F_{\delta}:z\to w$ be the holomorphic correspondence defined by $R_0(w)+\delta P(z,w)=0$ for any $\delta\geq 0$. For any $\epsilon>0$ there is $\Delta>0$ such that if $|\delta|<\Delta$, there is a probability measure $\mu_{F_\delta}$ on $\setRS$ with $(F_\delta)_*(\mu_{F_\delta})=\mu_{F_\delta}d$ such that for any $a\in \setRS$,

$$\frac{1}{d^m}(F_\delta^m)_*(\delta_a) \xrightarrow{weak *}\mu_{F_\delta}, \quad \text{ as } m\to \infty.$$
Additionally, the support $\supp \mu_{F_{\delta}}$ is contained in the $\epsilon$-neighborhood of the zeros of $R_0(z)$. Furthermore, $\supp \mu_{F_{\delta}}$  is the minimal closed under $F_{\delta}$ forward invariant set.

\end{proposition}

\begin{proof}
It suffices to prove equicontinuity of $\{A^m\phi\}_{m=1}^\infty$ at $\infty$, the rest follows precisely the same way as in the proof of \cref{thm:A}. Denote the degree in $z$ of the coefficient of the leading term in $u$ by $d_1>1$. There are $d_1$ (counting multiplicity) points in $F_\delta(\infty)$ different from $\infty$. Take $\epsilon_0>0$ and set $M=\|\phi\|_{\setRS}$. Find $l$ large enough so that $\left(\frac{d-d_1}{d}\right)^l<\frac{\epsilon_0}{4M}$. Denote the $d^l-(d-d_1)^l$ (counting multiplicity) points in $(F_\delta)^l(\infty)$ by $\xi_i$. We have $\xi_i\in \bC$ so we can in the same way as in the proof of \cref{thm:A} find $r>0$ such that if $d(\xi_i,\zeta)<r$ then 
$$|A^m\phi(\xi_i)-A^m\phi(\zeta)|\leq \epsilon_0/2$$
for all $m.$
Next, we can find $\eta$ such that if $d(\infty,\zeta)<\eta$ then $$d(\xi_i,\zeta_i)<r$$
for some points $$\zeta_i\in (F_\delta)^l(\zeta).$$
Summarizing we obtain that if $d(\infty,\zeta)<\eta$, then

$$|A^{l+m}\phi(\infty)-A^{l+m}\phi(\zeta)|$$
$$ \leq\frac{1}{d^l}\left(2M(d-d_1)^l+\sum_{i=1}^{d^l-(d-d_1)^l}|A^m\phi(\xi_i)-A^m\phi(\zeta_i)|\right) <\epsilon_0$$
for all $m\geq 0$. As all the members of the finite family $\{A^m\phi\}_{m=1}^{l-1}$ are continuous, the family itself is clearly equicontinuous at $\infty$. The conclusion of the first part of the proposition now follows in the same way as in \cref{thm:A}.
We have
$$(F_\delta)_*(\mu_{F_\delta})=d\mu_{F_\delta}.$$ In particular, $\supp \mu_{F_\delta}$ is closed and forward invariant. Let $S$ be a closed and non-empty forward invariant set under $F_\delta$. We have that for any $a\in S$, $$\frac{1}{d^m}(F_\delta^m)_*(\delta_a) \xrightarrow{weak *}\mu_{F_\delta}$$
so $\supp(\mu_{F_\delta})$ is contained in $S$. It follows that $\supp \mu_{F_\delta}$  is the minimal closed under $ {F_\delta}$ forward invariant set.
This concludes the proof.

\end{proof}

 \section{Proof of \cref{thm:C}}\label{sec:proofsb}

We recall that $T_n:z\to w$ is the holomorphic correspondence defined by $\frac{T[(w-z)^n]}{(w-z)^{n-k}}=0$, where $T$ is a linear operator and that if there is a minimal set $S$ with the property that it is closed, contains a point in $\bC$ and is forward invariant set under $T_n$, then $S\setminus {\infty}\subset \bC$ is equal to $\minvset{H,n}$.
\begin{lemma}
If $T$ is non-degenerate, then $T_n(\bC)\subset\bC$ for sufficiently large $n$.
\end{lemma}
\begin{proof}
$\frac{T[(w-z)^n]}{(w-z)^{n-k}}$ is a polynomial with leading term in $w$ of degree $\deg Q_k(w)$ for sufficiently large $n$ and the coefficient of the leading term for large $n$ is a non-zero constant. 
\end{proof}

We have the following useful implicit characterization of $\minvset{H,n}$. When we take the closure below, we do it in $\bC$.
\begin{lemma}\label{th:implicit}
Suppose that $\minvset{H,n}$ exists and $T_n(\bC)\subset \bC$. Then, for any $u\in \minvset{H,n}$, $\minvset{H,n}=\overline{\cup_{j=0}^\infty T^j_n(u)}.$
\end{lemma}
\begin{proof}
Since $\minvset{H,n}$ is closed and forward invariant under $T_n$, for any $z\in \minvset{H,n}$, we have that $\overline{\cup_{j=0}^\infty T^j_n(z)}\subset \minvset{H,n}$. 
We need to prove that $\overline{\cup_{j=0}^\infty T^j_n(z)}$ is $T_{H,n}$-invariant. Take $w\in \overline{\cup_{j=0}^\infty T^j_n(z)}$. We shall prove that $T_n(w)\in \overline{\cup_{j=0}^\infty T^j_n(z)}$. Suppose first $w\in \cup_{j=0}^\infty T^j_n(z)$. Then $w\in  T^K_n(z)$ for some finite $K$ and $T_n(w)\subset T^{K+1}_n(z)\subset \overline{\cup_{j=0}^\infty T^j_n(z)}$. Hence, suppose $w\in \overline{\cup_{j=0}^\infty T^j_n(z)}\setminus \cup_{j=0}^\infty T^j_n(z)$. Then there is a sequence $(z_m)\subset \cup_{j=0}^\infty T^j_n(z)$ with limit $w$. Moreover, $T_n(z_m)\in \cup_{j=0}^\infty T^j_n(z)$ for each $m$. Since the zero-locus of a polynomial is continuous in the coefficients, for every $w_0\in T_n(w)$ there is a sequence $(x_m)$ such that  $
x_m\in T(z_m)$ and $\lim_{m\to \infty} x_m =w_0$. In particular, $T_n(w)\subset \overline{\cup_{j=0}^\infty T^j_n(z)}$.
\end{proof}

We can now give the following proof.

 \begin{proof}[Proof of \cref{thm:C}]
Suppose $T=\sum_{j=0}^kQ_j(w)\frac{d^j}{dw^j}$ is non-degenerate and $Q_k(w)$ has only simple zeros. The solutions to $$\frac{T(w-z)^n}{(w-z)^{n-k}}=0$$
are given by the solutions to
$$\frac{T(w-z)^n}{(n)_k(w-z)^{n-k}}=Q_k(w)+\sum_{j=0}^{k-1}\frac{(n)_j}{(n)_k}(w-z)^{k-j}Q_j(w)=0.$$ 
Hence, $\frac{T(z-w)^n}{(n)_k(w-z)^{n-k}}$ can be written as written as $R_\beta$ in \cref{thm:A} above with $R_{0}(w)=Q_k(w)$, $R_{d}(w)\equiv 0$, $R_{d-j}(w)=Q_{k-j}(w)$ for $j=1,...,k$ , $R_j(w)\equiv 0$ for $j<k-M$, and $\beta_{k-j}=\frac{(n)_{k-j}}{(n)_k}$. Since $T$ is non-degenerate we have $\deg R_j(w)\leq j$. We have that for $j<k$, $\frac{(n)_{k-j}}{(n)_k}\to 0$ as $n\to \infty$. Hence, taking $n$ large enough so that $\frac{(n)_{j-1}}{(n)_k}\leq \frac{(n)_{k-1}}{(n)_k}<\Delta$ proves that $\minvset{H,n}$ exists and is the support of the measure $\mu_{T_n}$. 

To complete the proof, we will need the following lemma. 

\begin{lemma}\label{le:perfect}
Suppose that $T$ is non-degenerate, of order $k$ at least 2 and that there is a simple zero $u_1$ of $Q_k(w)$ such that  not all $Q_j(w)$ have $u_1$ as a common zero. Then $\minvset{H,n}$ is a perfect set for $n$ sufficiently large.
\end{lemma}
\begin{remark}
If $k\geq 2$ and all zeros of $Q_k$ are simple, $k\geq 2$ and at least one $Q_j(w)$, $j\neq k$ is not identically zero then, the hypothesis is fulfilled.
\end{remark}
\begin{proof}
Let $M$ be as in \cref{le:1}, $n$ be large and $u_1$ a simple zero of $Q_k(w)$. 
We will prove that if $z\in D(0,M), z\neq u_1$, then for all sufficiently large $n$, $u_1\notin T_n(z)$. To that end, suppose the contrary and let $z$ be such. Then $|u_1-z|\leq 2M$. We have that

$$\sum_{j=0}^{k-1}(n)_j Q_j(u_1)(u_1-z)^{k-j}=0,$$
which gives $$|Q_{J-1}(u_1)|\leq  \sum_{j=0}^{J-2}\left|\frac{(n)_j}{(n)_{J-1}} Q_j(u_1)(2M)^{J-j-1}\right|$$ where $J-1$ is the maximal value of $j$ such that $Q_j(u_1)\neq 0$, and the sum is 0 if $J = 1$. However, for large $n$, this is clearly false as the right-hand side tends to $0$ as $n\to \infty$, whereas the left-hand side does not. Hence, we can find a sequence $(u_l)_{l=2}^\infty$ such that $u_2\neq u_1$ $u_2\in T_n(u_1)$ and $u_{l+1}\in T_n^l(u_1)$, $u_l\neq u_1$ for $l\geq 2$ and $\lim_{l\to \infty}{u_l}=u_1$ by the proof of \cref{thm:A}.

Suppose now that $\minvset{H,n}$ is not a perfect set. As $\minvset{H,n}$ is closed, there then is an isolated point $w_0\in \minvset{H,n}$. As $w_0$ is an isolated point, then by \cref{th:implicit} ($n$ can be chosen large enough so that $T_n(\bC)\subset \bC$), $w_0$ belongs to $\cup_{j=0}^\infty T^j_n(u_1)$. Hence, there is some positive integer $K$ such that $w_0\in T^K_n(u_1)$. As we have seen above, there is a sequence of points $(u_l)^\infty_{l=2}$ not containing $u_1$ that converges to $u_1$. As the set of zeros of a polynomial is continuous in the coefficients, it follows that there is a sequence of points $(v_j)$ with $v_j\in T_n^K(u_j)$ converging to $z_0$; we may choose $v_j=\min_{v\in T_n^K(u_j)}|v_j-w_0|$. Now, there are at most $k^K$ distinct values of $z$ such that $w_0\in T_n^K(z)$. Indeed,

$$T[(w-z)^n]/(w-z)^{n-k}$$ 
with $w=w_0$ is a polynomial in $z$ of degree $k^K$.  Hence, at most finitely many of the $v_j$ are  equal to $w_0$. Therefore, there is a subsequence $(v_j)$ that does not contain $w_0$ and converges to $w_0$. This contradicts the assumption that $w_0$ is isolated. We conclude that $\minvset{H,n}$ is a perfect set.
\end{proof}

Once again, $w_l$, $l=1,..,d$  denote  the $d$ branches of $T_n(z)$ for $z\in D(0,M)$. Take now $n$ small enough so that $|w_l'(z)|<\frac{c}{d}$ for some $|c|<1$ for all $z\in \cup_j D(u_j,\eta_0)$. Suppose that there is a connected subset $C$ of $\supp(\mu_{T_n})$. Denote its diameter with $diam(C)$. We have that $\supp(\mu_{T_n})\subset (f_{n})^m(\cup_j D(u_j,\eta_0))$. We have that $(f_{n})^m(\cup_j D(u_j,\eta_0))$ is contained in the union of $d^m$ balls of radius $\eta_0\left(\frac{c}{d}\right)^m$. We pick $m$ large enough so that $2\eta_0 c^m<diam(C)$. This means that the diameter of any connected component of $(f_{n})^m(\cup_j D(u_j,\eta_0))$ has diameter at most $2\eta_0 c^m<diam(C),$ which contradicts the existence of $C$. We conclude that $\supp(\mu_{F_n})=\minvset{H,n}$ is totally disconnected and by \cref{le:perfect} it is a perfect set. We have that $\minvset{H,n}$ clearly is compact and metrizable as it is a closed subset of $\setRS$. It follows that it is a Cantor set. 

Let us now prove that each point in $\minvset{H,n}$ is the limit of a sequence of periodic points of $T_n$. For sufficiently large $n$, any periodic point of a $w_{I}$ in $D(0,M)$ must be attracting by the chain rule and the fact that $w_s'(u)<1/2$ for $n$ large.

Take $u_1\in \zeros(Q_k)$, $w_1(z)$ be the local branch of $T_n $ in $D(0,M)$ that fixes $u_1$ and let $w\in \minvset{H,n}$ be given. Suppose first that $w= w_{I}(u_1)$ for some multi-index $I$. Consider the sequence of maps $w_I\circ w_1^\eta$, $\eta\in \mathbb N$. They all map $D(0,M)$ to $D(0,M)$ so by Brouwer's fixed point theorem they each have a fixed point $u_\eta$ in $D(0,M).$ For any $\eta>0$ there is an $\epsilon_\eta>0$ such that $\epsilon_\eta\to 0$ and $D(0,M)$ is mapped to the $\epsilon_\eta$-neighborhood of $z$. In particular, $u_\eta\to w$.

Now suppose $w$ is the limit of points of the form $z_{I_j}(u_1)$ where the length of $I_j$, denoted by $\eta_j$ tends to $\infty$ as $j\to \infty$. After passing to a subsequence, we can suppose that $\eta_j$ is increasing. Let $\epsilon_1>0$ be given. Take $N_0$ to be such that $1/\eta_j<\epsilon_1/2$ if $j\geq N_0$ . For each $\eta_j\geq 1$, we can, using continuity of the $w_l$ and the fact that $|w_1'(z)|<1/2$ for $z\in D(0,M)$, find $K(\eta_j)$ such that diam$(w_{I_j}\circ w_{1}^{K(\eta_j)}(D(0,M)))< 1/\eta_j$. By the definition of $w$, we can also find $N_1$ large enough so that $|w-w_{I_j}(u_1)| < \epsilon_1/2$ for any $j \geq N_1$. Now, there is a fixed point of $w_{I_{j}}\circ w_0^{K(\eta_j)}$ contained in $w_{I_{j}}\circ w_0^{K(\eta_j)}(D(0,M))$. Moreover, $w_{I_j}(u_1) \in w_{I_{j}}\circ w_0^{K(\eta_j)}(D(0,M)) $  since $w\in D(0,M)$ and $w_1$ fixes $u_1$. Hence, by Brouwer's fixed point theorem there is a sequence of points $v_{I_j,K(\eta_j)}$ which are fixed points of $w_{I_{j}}\circ w_1^{K(\eta_j)}$ such that if $j\geq \max(N_0,N_1)$, then $|w-v_{I_j,K(\eta_j)}| < \epsilon_1$. It follows that $\lim_{j\to \infty} v_{I_j,K(\eta_j)}=w$. Thus, $w$ is a limit of a sequence of attracting fixed points in the two cases above. By \cref{th:implicit} these are the only cases to consider, and we have thereby proven the result.

This concludes the proof of \cref{thm:C}.

\end{proof}

Using precisely the same methods as above, we may prove the following statement, which we hinted at in \cref{sec:intro}.

\begin{proposition}\label{prop:prop1}
Suppose that for a non-degenerate $T$ given by \eqref{eq:1}, $Q_k(w)$ has a simple zero $u_0$. For any $\epsilon>0$, there is $N$ such that if $n\geq N$, then
\begin{itemize}\item $\minvset{H,n}$ exists, is contained in the $\epsilon$-neighborhood of the zeros of $Q_k$ and contains the simple zeros of $Q_k(w)$. \item If $k\geq 2$ and at least one $Q_j(w)$, $j\neq k$ is not identically zero then $\minvset{H,n}$ is a perfect set.\end{itemize}
\end{proposition}

\section{Final Remarks} \label{sec:final}

This paper barely scratches the surface of the study of invariant sets under holomorphic correspondences and there are multiple natural questions open. We name merely a few of the questions left unanswered in this paper.

\begin{problem}
Give necessary and sufficient conditions for when $\minvset{H,n}$ exists.
\end{problem}
\begin{problem}
In this paper, we analyzed the case when $n$ is large. What can be said for small $n$?
\end{problem}
\begin{problem}
Show that for large $n$, any zero of $Q_k$ belongs to $\minvset{H,n}$.
\end{problem}

\bibliographystyle{alphaurl}
\typeout{}
\bibliography{./theBibliography}

\end{document}